\documentclass[12pt,a4paper]{article}

\usepackage[T1]{fontenc}
\usepackage[utf8]{inputenc}
\usepackage[english]{babel}

\usepackage{mathrsfs}
\usepackage{amsmath}
\usepackage{amssymb}
\usepackage{amsthm}
\usepackage[colorlinks=true, linkcolor=black, citecolor=black, urlcolor=blue]{hyperref}
\usepackage{microtype}

\newtheorem{Theorem}{Theorem}
\newtheorem{Lemma}{Lemma}
\newtheorem{Proposition}{Proposition}
\newtheorem{Definition}{Definition}

\title{
	On the Strong Law of Large Numbers under Sublinear Expectations
}

\author{I.V. Kozlov\footnote{M.V. Lomonosov Moscow State University; Institute for Information Transmission Problems (A.A. Kharkevich Institute); Vega Institute Foundation; email: ilia.kozlov@math.msu.ru}
}

\date{}

\begin{document}
	\maketitle
	
	\begin{abstract}
		\noindent
		A strong law of large numbers (SLLN) for sublinear expectation under a certain integrability condition is proved. It is shown that the proposed integrability condition is weaker than the majorant condition introduced by Cheng Hu (2018). The proof is based on the truncation method of N. Etemadi (1981).
		
		\medskip
		
		\noindent
		Keywords: Strong law of large numbers; sublinear expectation.
		
		\medskip
		
		\noindent
		MSC2020: 60F15 \\
		UDK: {519.214}
		
	\end{abstract}
	
	\section{Introduction}
	
	Sublinear expectation theory, developed by Shige Peng (see \cite{Peng97, Peng07, Peng10}), provides a mathematical framework for stochastic models under distributional uncertainty. It often appears in stochastic control, risk measures, and the theory of backward stochastic differential equations. An important question in this context is whether classical limit theorems extend to the setting of nonlinear expectations. Peng also proved the first law of large numbers in this context. After his work, other researchers proposed several versions of the strong law of large numbers for sublinear expectations. However, existing results require quite restrictive integrability assumptions. In particular, C. Hu \cite{Hu18} established the SLLN under the assumption that the sequence admits an integrable majorant.
	
	We prove a new version of the SLLN for sublinear expectations using weaker, more general conditions. Instead of requiring a global majorant, our new condition uses the supremum of integrals over a family of probability measures. Our proof follows the truncation argument introduced by Etemadi \cite{Etemadi81} in the classical setting. We show that this method can be adapted to sublinear expectations by combining truncation with the Borel--Cantelli lemma for capacities. Lastly, we construct an example demonstrating that our condition is strictly weaker than Hu's majorant condition \cite{Hu18}.
	
	\section{Basic Concepts and Statements}
	We use the standard notation introduced by Peng \cite{Peng97, Peng07, Peng08, Peng10}. Let $(\Omega, \mathcal{F})$ be a given measurable space. Let $\mathcal{H}$ be a subset of all random variables defined on $(\Omega, \mathcal{F})$ such that $I_M \in \mathcal{H}$ for all $M \in \mathcal{F}$, and if $X_1, X_2, \ldots, X_n \in \mathcal{H}$, then $\varphi(X_1, X_2, \ldots, X_n) \in \mathcal{H}$ for each $\varphi \in C_{l,\mathrm{Lip}}(\mathbb{R}^n)$, where $C_{l,\mathrm{Lip}}(\mathbb{R}^n)$ denotes the linear space of locally Lipschitz functions $\varphi$ satisfying
	\[
	|\varphi(x) - \varphi(y)|
	\leq C\bigl(1 + |x|^m + |y|^m\bigr)|x - y|,
	\qquad \forall\, x, y \in \mathbb{R}^n,
	\]
	where $C > 0$ and $m \in \mathbb{N}$ depend on $\varphi$.
	
	\begin{Definition}
		\label{def:2.1}
		A \textit{sublinear expectation} $\mathbb{E}$ on $\mathcal{H}$ is a functional $\mathbb{E}\colon \mathcal{H} \to \overline{\mathbb{R}} := [-\infty, +\infty]$ satisfying the following properties for all $X, Y \in \mathcal{H}$:
		\begin{enumerate}
			\item[(a)] \textit{Monotonicity.} If $X \geq Y$, then $\mathbb{E}[X] \geq \mathbb{E}[Y]$.
			\item[(b)] \textit{Constant preserving.} $\mathbb{E}[c] = c$, $\forall\, c \in \mathbb{R}$.
			\item[(c)] \textit{Positive homogeneity.} $\mathbb{E}[\lambda X] = \lambda\,\mathbb{E}[X]$, $\forall\, \lambda \geq 0$.
			\item[(d)] \textit{Subadditivity.} $\mathbb{E}[X + Y] \leq \mathbb{E}[X] + \mathbb{E}[Y]$, provided that $\mathbb{E}[X] + \mathbb{E}[Y]$ does not take the form $\infty - \infty$ or $-\infty + \infty$.
		\end{enumerate}
	\end{Definition}
	
	\begin{Definition}
		\label{def:2.2}
		A set function $V : \mathcal{F} \to [0,1]$ is called a \textit{capacity} if it satisfies:
		\begin{enumerate}
			\item[(a)] $V(\emptyset) = 0$, $V(\Omega) = 1$.
			\item[(b)] $V(A) \leq V(B)$ for $A \subset B$, $A, B \in \mathcal{F}$.
		\end{enumerate}
		A capacity $V$ is called \textit{sub-additive} if $V(A \cup B) \leq V(A) + V(B)$ for all $A, B \in \mathcal{F}$.
	\end{Definition}
	
	Throughout this paper, we consider only capacities induced by a sublinear expectation. Let $(\Omega, \mathcal{H}, \mathbb{E})$ be a sublinear expectation space. We define the capacity: $V(A) := \mathbb{E}[I_A]$, $\forall A \in \mathcal{F}$, and the conjugate capacity: $v(A) := 1 - V(A^c)$, $\forall A \in \mathcal{F}$. It is easy to see that $V$ is a sub-additive capacity and $v(A) = \mathcal{E}[I_A]$, where $\mathcal{E}[X] := -\mathbb{E}[-X]$.
	
	\begin{Definition}
		\label{def:2.3}
		A sublinear expectation $\mathbb{E} : \mathcal{H} \to \mathbb{R}$ is said to be \textit{continuous} if it satisfies:
		\begin{enumerate}
			\item[(1)] \textit{Lower-continuity:} If $X_n \uparrow X$, then
			$\mathbb{E}[X_n] \uparrow \mathbb{E}[X]$, where $0 \le X_n$, $X \in \mathcal{H}$.
			
			\item[(2)] \textit{Upper-continuity:} If $X_n \downarrow X$, then
			$\mathbb{E}[X_n] \downarrow \mathbb{E}[X]$, where $0 \le X_n$, $X \in \mathcal{H}$.
		\end{enumerate}
		
		A capacity $V : \mathcal{F} \to [0,1]$ is called a \textit{continuous capacity} if it satisfies:
		\begin{enumerate}
			\item[(1)] \textit{Lower-continuity:} If $A_n \uparrow A$, then
			$V(A_n) \uparrow V(A)$, where $A_n, A \in \mathcal{F}$.
			
			\item[(2)] \textit{Upper-continuity:} If $A_n \downarrow A$, then
			$V(A_n) \downarrow V(A)$, where $A_n, A \in \mathcal{F}$.
		\end{enumerate}
	\end{Definition}
	
	Throughout this paper, we assume that the representing family of probability measures $\mathcal{P}$ consists of $\sigma$-additive measures on $(\Omega, \mathcal{F})$ (see \cite{Peng10}). The following proposition summarizes these representation results.
	
	\begin{Proposition}
		\label{prop:2.2}
		Let $(\Omega, \mathcal{H}, \mathbb{E})$ be a sublinear expectation space.
		\begin{enumerate}
		\item[(1)] There exists a family of countably additive probability measures $\{P_\theta : \theta \in \Theta\}$ on $(\Omega, \mathcal{F})$ such that for each $X \in \mathcal{H}$,
			\[
			\mathbb{E}[X] = \sup_{\theta \in \Theta} E_{P_\theta}[X].
			\]
			\item[(2)]  For any fixed $X \in \mathcal{H}$, there exists a family of probability measures $\{\mu_\theta\}_{\theta \in \Theta}$ on $(\mathbb{R}, \mathcal{B}(\mathbb{R}))$ such that for each $\varphi \in C_{l,\mathrm{Lip}}(\mathbb{R})$,
			\[
			\mathbb{E}[\varphi(X)] = \sup_{\theta \in \Theta} \int_{\mathbb{R}} \varphi(x)\,\mu_\theta(\mathrm{d}x).
			\]
		\end{enumerate}
	\end{Proposition}
	
	\begin{Definition}
		\label{def:2.4}
		Given a capacity $V$, a set $A$ is called \textit{polar} if $V(A) = 0$. A property is said to hold \textit{quasi-surely} (q.s.) if it holds outside a polar set.
	\end{Definition}
	
	\begin{Definition}
		\label{def:2.5}
		(Independence) Let $\boldsymbol{X} = (X_1, \ldots, X_m)$, $X_i \in \mathcal{H}$, and $\boldsymbol{Y} = (Y_1, \ldots, Y_n)$, $Y_i \in \mathcal{H}$ be two random vectors on $(\Omega, \mathcal{H}, \mathbb{E})$. $\boldsymbol{Y}$ is \textit{independent of} $\boldsymbol{X}$ if for each function $\varphi \in C_{l,\mathrm{Lip}}(\mathbb{R}^m \times \mathbb{R}^n)$,
		\[
		\mathbb{E}[\varphi(\boldsymbol{X}, \boldsymbol{Y})] = \mathbb{E}\bigl[\mathbb{E}[\varphi(\boldsymbol{x}, \boldsymbol{Y})] \big|_{\boldsymbol{x} = \boldsymbol{X}}\bigr]
		\]
		provided that $\bar{\varphi}(\boldsymbol{x}) := \mathbb{E}[|\varphi(\boldsymbol{x}, \boldsymbol{Y})|] < \infty$ for all $\boldsymbol{x}$ and $\mathbb{E}[|\bar{\varphi}(\boldsymbol{X})|] < \infty$. A sequence $\{X_n\}_{n=1}^{\infty}$ is called a \textit{sequence of independent random variables} if $X_{n+1}$ is independent of $(X_1, \ldots, X_n)$ for each $n \geq 1$.
	\end{Definition}
	
	\begin{Definition}
		\label{def:2.6}
		(Identical distribution) Let $\boldsymbol{X}_1$ and $\boldsymbol{X}_2$ be two $n$-dimensional random vectors defined on $(\Omega_1, \mathcal{H}_1, \mathbb{E}_1)$ and $(\Omega_2, \mathcal{H}_2, \mathbb{E}_2)$ respectively. They are called \textit{identically distributed} if
		\[
		\mathbb{E}_1[\varphi(\boldsymbol{X}_1)] = \mathbb{E}_2[\varphi(\boldsymbol{X}_2)], \qquad \forall\, \varphi \in C_{l,\mathrm{Lip}}(\mathbb{R}^n),
		\]
		provided the corresponding sublinear expectations are finite.
	\end{Definition}
	
	The following lemmas are needed in the sequel. Their proofs are available in~\cite{Chen13}.
	
	\begin{Lemma}[Borel--Cantelli lemma]
		\label{lem:2.1}
		Let $\{A_n\}_{n=1}^{\infty}$ be a sequence of events in $\mathcal{F}$, and let $V$ be a capacity induced by a lower-continuous sublinear expectation $\mathbb{E}$. If
		\[
		\sum_{n=1}^{\infty} V(A_n) < \infty,
		\]
		then
		\[
		V\!\left(\bigcap_{n=1}^{\infty}\bigcup_{i=n}^{\infty} A_i\right)=0.
		\]
	\end{Lemma}
	
	\begin{Lemma}
		\label{lem:2.2}
		\textnormal{(Chebyshev's inequality)}
		Let $f(x) > 0$ be a non-decreasing function on $\mathbb{R}$. Then for any $x$,
		\[
		V(X \geq x) \leq \frac{\mathbb{E}[f(X)]}{f(x)}, \qquad v(X \geq x) \leq \frac{\mathcal{E}[f(X)]}{f(x)}.
		\]
	\end{Lemma}
	
	\section{Main Theorem}
	
	\begin{Theorem}\label{thm:main_slln}
		Given a sublinear expectation space $(\Omega, \mathcal{H}, \mathbb{E})$ where $\mathbb{E}$ is lower-continuous and $V$ is the induced capacity, let $\mathcal{P}$ be the family of probability measures representing $\mathbb{E}$. Let $\{X_n\}_{n=1}^{\infty}$ be a sequence of independent and identically distributed non-negative random variables in $\mathcal{H}$. Suppose that the following integrability condition holds:
		\[
		\sum_{k=0}^{\infty} \sup_{Q \in \mathcal{P}} \int_k^{k+1} x \, \mathrm{d}F_{Q}(x) < \infty,
		\]
		where $F_{Q}(x) = Q(X_1 \le x)$ is the distribution function of $X_1$ under $Q$. 
		
		Then, for $S_n = \sum_{i=1}^n X_i$, we have
		\[
		V \left( \limsup_{n \to \infty} \frac{S_n}{n} > \mathbb{E}[X_1] \ \cup \ \liminf_{n \to \infty} \frac{S_n}{n} < -\mathbb{E}[-X_1] \right) = 0.
		\]
	\end{Theorem}
	
	\begin{proof}
		\textbf{Proof of the upper bound of the SLLN:}
		We use Etemadi's truncation method. Define $Y_i := X_i \cdot I_{\{X_i \le i\}}$ and let $S_n^* := \sum_{i=1}^n Y_i$. Also, for a fixed $\alpha > 1$, we set $k_n := \lfloor \alpha^n \rfloor$.
		
		By Chebyshev's inequality (Lemma \ref{lem:2.2}) and Rosenthal's inequality for independent random variables under sublinear expectations (see \cite{Zhang16}), there exists a constant $c > 0$ such that:
		\begin{align*}
			\sum_{n=1}^{\infty} V \left\{ \frac{|S_{k_n}^* - \mathbb{E}S_{k_n}^*|}{k_n} > \varepsilon \right\} 
			&\le \frac{1}{\varepsilon^2} \sum_{n=1}^{\infty} \frac{1}{k_n^2} \mathbb{E} \left[ \left( \sum_{i=1}^{k_n} (Y_i - \mathbb{E}Y_i) \right)^2 \right] \\[10pt]
			&\le c \sum_{n=1}^{\infty} \frac{1}{k_n^2} \sum_{i=1}^{k_n} \mathbb{E}Y_i^2 \\[10pt]
			&= c \sum_{i=1}^{\infty} \mathbb{E}[Y_i^2] \sum_{n: k_n \ge i} \frac{1}{k_n^2} \\[10pt]
			&\le c \sum_{i=1}^{\infty} \frac{1}{i^2} \mathbb{E}[Y_i^2] \\[10pt]
			&= c \sum_{i=1}^{\infty} \frac{1}{i^2} \sup_{Q \in \mathcal P} \int_0^i x^2 \mathrm{d} F_{Q}(x) \\[10pt]
			&\le c \sum_{i=1}^{\infty} \frac{1}{i^2} \left( \sum_{k=0}^{i-1} \sup_{Q \in \mathcal P} \int_k^{k+1} x^2 \mathrm{d} F_{Q}(x) \right) \\[10pt]
			&\leq c \sum_{k=0}^{\infty} \frac{1}{k+1}  \sup_{Q \in \mathcal P}  \int_k^{k+1} x^2 \mathrm{d} F_{Q}(x) \\[10pt]
			&\leq c \sum_{k=0}^{\infty} \frac{1}{k+1}  \sup_{Q \in \mathcal P}  \int_k^{k+1} (k+1) x \mathrm{d} F_{Q}(x) \\[10pt] 
			&= c \sum_{k=0}^{\infty} \sup_{Q \in \mathcal P}  \int_k^{k+1} x \mathrm{d} F_{Q}(x)
			< \infty.
		\end{align*}
		
		By subadditivity:
		\[
		\frac{\mathbb{E}S_{k_n}^*}{k_n} = \frac{1}{k_n} \mathbb{E}\left[ \sum_{i=1}^{k_n} Y_i \right] \leq \frac{1}{k_n} \sum_{i=1}^{k_n} \mathbb{E}Y_i.
		\]
		Since the random variables are identically distributed, we have $\mathbb{E}[Y_n] = \mathbb{E}[X_1 \cdot I_{\{X_1 \le n\}}]$. 
		Since $X_1 \geq 0$, the sequence $X_1 \cdot I_{\{X_1 \le n\}} \uparrow X_1$ monotonically as $n \to \infty$. Therefore, by the lower-continuity of the sublinear expectation $\mathbb{E}$, we obtain
		\[
		\lim_{n \to \infty} \mathbb{E}[Y_n] = \mathbb{E}[X_1].
		\]
		By Cesaro's lemma, the arithmetic mean converges to the same limit:
		\[
		\limsup_{n \to \infty} \frac{\mathbb{E}S_{k_n}^*}{k_n} \leq \lim_{n \to \infty} \frac{1}{k_n} \sum_{i=1}^{k_n} \mathbb{E}Y_i = \mathbb{E}X_1.
		\]
		Therefore, by the Borel--Cantelli lemma:
		\[
		\limsup_{n \to \infty}\frac{ S_{k_n}^*}{k_n} \leq \mathbb{E} [X_1] \quad \text{q.s.}
		\]
		
		Next, we estimate the capacity of the event $\{Y_n \neq X_n\}$:
		\begin{align*}
			\sum_{n=1}^{\infty} V\{ Y_n \neq X_n \} &= \sum_{n=1}^{\infty} V\{ X_n > n \} = \sum_{n=1}^{\infty}  \sup_{Q \in \mathcal P} Q(X_1 > n) \\[10pt]
			&\leq \sum_{n=1}^{\infty} \sum_{k=n}^{\infty} \sup_{Q \in \mathcal P} Q(k < X_1 \le k+1) \\[10pt]
			&= \sum_{k=1}^{\infty} \sum_{n=1}^{k} \sup_{Q \in \mathcal P} Q(k < X_1 \le k+1) \\[10pt]
			&= \sum_{k=1}^{\infty} k \sup_{Q \in \mathcal P} \int_k^{k+1} \mathrm{d} F_Q(x) \\[10pt]
			&\leq \sum_{k=1}^{\infty}  \sup_{Q \in \mathcal P}  \int_k^{k+1} x \mathrm{d} F_Q(x) < \infty.
		\end{align*}
		
		Thus, by the Borel--Cantelli lemma, $X_n \neq Y_n$ for only finitely many $n$ q.s. Hence,
		\[
		\limsup_{n \to \infty} \frac{S_{k_n}}{k_n} \leq \mathbb{E} [X_1] \quad \text{q.s.}
		\]
		
		Now, taking into account the monotonicity of $S_n$ (since $X_i \ge 0$), for any $m$ such that $k_n \le m < k_{n+1}$, we can deduce that:
		\[
		\frac{S_{k_{n}}}{k_{n+1}} \leq \frac{S_{m}}{m}  \leq \frac{S_{k_{n+1}}}{k_n}.
		\]
		Since $k_{n+1}/k_n \to \alpha$ as $n \to \infty$, the right-hand inequality can be rewritten as:
		\[
		\frac{S_m}{m} \leq \frac{S_{k_{n+1}}}{k_{n+1}} \cdot \frac{k_{n+1}}{k_n}.
		\]
		Taking the limit superior:
		\[
		\limsup_{m \to \infty} \frac{S_m}{m} \le \alpha \limsup_{n \to \infty} \frac{S_{k_{n+1}}}{k_{n+1}} \le \alpha \mathbb{E} [X_1] \quad \text{q.s.}
		\]
		Since this inequality holds quasi-surely for any $\alpha > 1$, letting $\alpha \downarrow 1$ along a countable sequence, we obtain the upper bound of the SLLN:
		\[
		V \left(\limsup_{n \to \infty} \frac{S_n}{n} > \mathbb{E} [X_1] \right) = 0.
		\]
		
		\textbf{Proof of the lower bound of the SLLN:}
		Consider the sequence $W_i = -X_i \leq 0$. Let $T_n = \sum_{i=1}^n W_i = -S_n$. 
		Applying the same truncation argument for $W_i$, by the Borel--Cantelli lemma for the subsequence $k_n = \lfloor \alpha^n \rfloor$, we get:
		\[
		\limsup_{n \to \infty} \frac{T_{k_n}}{k_n} \le \mathbb{E} [W_1] = \mathbb{E} [-X_1] \quad \text{q.s.}
		\]
		Since $W_i \le 0$, the sequence of sums $T_n$ is monotonically decreasing. Therefore, for any $m \in [k_n, k_{n+1})$, we have $T_{k_n} \ge T_m \ge T_{k_{n+1}}$. Dividing by $m$, we get:
		\[
		\frac{T_m}{m} \le \frac{T_{k_n}}{m} \le \frac{T_{k_n}}{k_{n+1}} = \frac{T_{k_n}}{k_n} \cdot \frac{k_n}{k_{n+1}}.
		\]
		Taking the limit superior with $k_{n+1}/k_n \to \alpha$:
		\[
		\limsup_{m \to \infty} \frac{T_m}{m} \le \frac{1}{\alpha} \limsup_{n \to \infty} \frac{T_{k_n}}{k_n} \le \frac{1}{\alpha} \mathbb{E} [-X_1] \quad \text{q.s.}
		\]
		Letting $\alpha \downarrow 1$ along a countable sequence, we obtain:
		\[
		\limsup_{n \to \infty} \frac{-S_n}{n} \le \mathbb{E} [-X_1] \quad \text{q.s.}
		\]
		Using the property $\limsup(-a_n) = -\liminf(a_n)$, we arrive at the lower bound:
		\[
		-\liminf_{n \to \infty} \frac{S_n}{n} \le \mathbb{E} [-X_1] \quad \implies \quad \liminf_{n \to \infty} \frac{S_n}{n} \ge -\mathbb{E} [-X_1] \quad \text{q.s.}
		\]
		Thus, we obtain
		\[
		V \left( \limsup_{n \to \infty} \frac{S_n}{n} > \mathbb{E} [X_1] \ \bigcup \ \liminf_{n \to \infty} \frac{S_n}{n} < -\mathbb{E} [-X_1] \right) = 0.
		\]
		The theorem is proved.
	\end{proof}
	
	\section{Comparison with Hu's Condition} 
	We present an example showing that our condition is weaker than the majorant condition by Cheng Hu \cite{Hu18} for a specific class of random variables.
	
	Consider $|X_n|$, $n = 1, 2, \dots$, where $X_n \sim \mathcal{N}(0, 1)$ is a sequence of i.i.d. random variables on a classical probability space ($\mathcal{P}=\{\mathsf P\}$). 
	
	\textbf{1. Convergence condition of the series:}
	For any $i \in \mathbb{N}$:
	\[
	\sum_{k=0}^\infty E_{\mathsf P} [|X_i| \cdot I_{\{k < |X_i| \le k+1\}}] = E_{\mathsf P}[|X_i|] = \sqrt{2/\pi} < \infty. 
	\]
	
	\textbf{2. Cheng Hu's majorant condition:}
	The condition $|X_n| \le |X|$ q.s. $\forall n$ requires $X \ge \sup_n |X_n|$. 
	Since $\mathsf P(|X_n| > M) > 0$ for any $M$, by the Borel--Cantelli lemma we have:
	\[
	\mathsf P\left(\sup_n |X_n| = \infty\right) = 1 \implies X = \infty \text{ q.s.} 
	\]
	Consequently, the limit of the tails is $\lim_{n \to \infty} \mathbb{E}[|X| I(|X| > n)] \neq 0$, as the majorant is not integrable.
	
	\section*{Acknowledgements}
	The author is supported by a scholarship from the Foundation for the Advancement of Theoretical Physics and Mathematics ``BASIS''.


\begin{thebibliography}{99}
		
		\bibitem{Peng97}
		S. Peng, Backward SDE and related $g$-expectation, Backward Stochastic Differential Equations, Pitman Research Notes in Math. Series, vol. 364, 1997, 141--160.
		
		\bibitem{Peng07}
		S. Peng, $G$-expectation, $G$-Brownian motion and related stochastic calculus of Ito type, Proceedings of the 2005 Abel Symposium, Springer-Verlag, Berlin, Heidelberg, 2007, 541--567.
		
		\bibitem{Peng10}
		S. Peng, Nonlinear expectations and stochastic calculus under uncertainty: with robust CLT and G-Brownian motion, Probability Theory and Stochastic Modelling, vol. 95, Springer, Berlin, Heidelberg, 2019.
		
		\bibitem{Hu18}
		C. Hu, Strong laws of large numbers for sublinear expectation under controlled 1st moment condition, Chinese Ann. Math. Ser. B, 2018, 39(5), 791--804.
		
		\bibitem{Etemadi81}
		N. Etemadi, An elementary proof of the strong law of large numbers, Z. Wahrscheinlichkeitstheorie verw. Gebiete, 1981, 55, 119--122.
		
		\bibitem{Peng08}
		S. Peng, Multi-dimensional $G$-Brownian motion and related stochastic calculus under $G$-expectation, Stoch. Proc. Appl., 2008, 118(12), 2223--2253.
		
		\bibitem{Chen13}
		Z. Chen, P. Wu, B. Li, A strong law of large numbers for non-additive probabilities, Int. J. Approx. Reason., 2013, 54(3), 365--377.
		
		\bibitem{Zhang16}
		L. Zhang, Rosenthal's inequalities for independent and negatively dependent random variables under sublinear expectation with applications, Sci. China Math., 2016, 59(4), 751--768.
		
	\end{thebibliography}
\end{document}